\documentclass[12pt]{amsart}
\usepackage{amsmath}           
\usepackage{amssymb}

\theoremstyle{plain}

\DeclareMathOperator{\lt}{L}

\DeclareMathOperator{\FM}{FM}

\DeclareMathOperator{\Q}{{\sf Q}}

\begin{document}

\newtheorem{thm}{Theorem}   
\newtheorem{lemma}[thm]{Lemma}   
\newtheorem{proposition}[thm]{Proposition}   
\newtheorem{theorem}[thm]{Theorem}   
\newtheorem{corollary}[thm]{Corollary}   
\newtheorem{pro}[thm]{Proposition}
\newtheorem{cor}[thm]{Corollary}
\newtheorem{lem}[thm]{Lemma}
\newtheorem{claim}[thm]{Claim}
\newtheorem{fact}[thm]{Fact}
\newtheorem{definition}[thm]{Definition}
\newtheorem{observation}[thm]{Observation}
\newcommand{\red}{\textcolor[rgb]{1.00,0.00,0.00}}

\newcommand{\lab}[1]{\label{#1}\marginpar{\footnotesize #1}} 
\renewcommand{\labelenumi}{{\textrm{(\roman{enumi})}}}
\newcommand{\vep}{\varepsilon} 
\newcommand{\One}{1} 
\newcommand{\Zero}{0} 
  \newcommand{\mb}[1]{\mathbb{#1}}
  \newcommand{\mc}[1]{\mathcal{#1}}
  \renewcommand{\phi}{\varphi}
\newcommand{\Qm}{\Q \mc{M}_f}
\newcommand{\et}{\wedge}
\newcommand{\imp}{\rightarrow  }
 \newcommand{\ex}{\exists  }
 \newcommand{\all}{\forall  }
\newcommand{\Mod}{{\sf Mod } }
\newcommand{\al  }{\alpha }
\newcommand{\be  }{\beta }
\newcommand{\ga  }{\gamma }
\newcommand{\de  }{\delta }
\newcommand{\Si}{\Sigma  }
\newcommand{\no}{\noindent}
\newcommand{\sub}{\subseteq}
\renewcommand{\th}{\theta}
\newcommand{\eqq}{\mb{E}}
\newcommand{\xc}[1]{}

\title[Finiteness problem]{On the Finiteness Problem
for classes of modular lattices}

\author[C. Herrmann]{Christian Herrmann}
\address{Technische Universit\"{a}t Darmstadt FB4\\Schlo{\ss}gartenstr. 7\\64289 Darmstadt\\Germany}
\email{herrmann@mathematik.tu-darmstadt.de}

\dedicatory{Dedicated to the memory of Rudolf Wille}


\begin{abstract}
The Finiteness Problem is
shown to be unsolvable for any sufficiently large class of modular lattices.
\end{abstract}

\maketitle

Given a class $\mc{A}$ of algebraic structures,
the \emph{Finiteness Problem}
is to decide for any given finite presentation, that is a list of
generator symbols and relations,
whether or not there is a finite bound on the size
of members of the class which  
'admit the presentation', that is 
a system of generators
satisfying the given relations; if $\mc{A}$ is
a quasi-variety, this means finiteness of the free $\mc{A}$-algebra
given by the presentation.
Due to Slavik \cite{slavik}, the finiteness problem is
algorithmically solvable for the class of all lattices,
due to Wille \cite{wille} for any class of modular lattices,
containing the subspace lattice of an infinite projective
plane, if one allows only order relations between the generators.
The present note relies on the unsolvability
of the Triviality Problem for modular lattices \cite{sat}
which in turn relies on the result of Adyan \cite{adyan,adyan2} and 
Rabin \cite{rabin} for groups.
For a vector space $V$ let $\lt(V)$ denote the lattice of subspaces.

\begin{theorem}\label{1}
Let $\mc{A}$ a class of 
modular lattices such that  $\lt(V)\in \mc{A}$
for some $V$ of infinite dimension.
Then the Finiteness Problem for 
$\mc{A}$ is algorithmically  unsolvable.
\end{theorem}

The following restates the relevant part of Lemma 10 in \cite{sat}.
\begin{lemma}\label{2}
There is a recursive set $\Sigma$ of conjunctions 
$\phi(\bar x,x_\bot,x_\top)$
of lattice equations 
such that 
 $\forall \bar x \forall x_\bot
\forall x_\top.\;\phi(\bar x,x_\bot,x_\top) \Rightarrow
\bigwedge_i   x_\bot \leq x_i \leq x_\top$
is valid in all modular lattices
and such that the  following hold
where  $\phi^\exists$ denotes the sentence
$\exists \bar x\exists x_\bot\exists x_\top.\;\phi(\bar x,x_\bot,x_\top)\wedge x_\bot\neq x_\top$. 
 \begin{enumerate}
\item 
 If, for    $\phi\in      \Sigma$, $\phi^\exists$
 is valid in  some  modular lattice, then
it is so  within 
 $\lt(V)$
for any $V$ of infinite dimension. Moreover, one can choose
$x_\bot=0$ and $x_\top=V$.
\item The set of all $\phi \in \Sigma$ 
with $\phi^\exists$ valid in some modular
lattice is not recursive.
\end{enumerate}
 \end{lemma}
Consider the conjunction $\pi(\bar y,y_\bot,y_\top)$
of the  following lattice equations
\[ y_i\cdot y_j=y_\bot \;(1\leq i<j\leq 4),\quad
y_i+y_j=y_\top\;(1\leq i<j\leq 4,\,j\neq 2)  \]
We use $x,y,\ldots$  both as  variables and generator
symbols and also to denote their values under 
a particular assignment.   
In  \cite{dhw}, $\FM(J_4^1)$
was defined as the modular lattice freely generated under the
presentation  $\pi(\bar y,y_\bot,y_\top)$ (equivalently, by 
the partial lattice $J^4_1$ arising from the 
$6$-element height $2$ lattice $M_4$ 
with atoms $y_1,y_2,y_3,y_4$ keeping all joins and meets 
except the join of $\{y_1,y_2\}$).
The following was shown
(to prove  (i) 
consider $V$ the direct sum of
infinitely many   subspaces of dimension $\aleph_0$).

\begin{lemma}\label{3}
Up to isomorphism, $M_4$ and singleton are the
 only proper homomorphic images of
$\FM(J_4^1)$.
Moreover, $\FM(J_4^1)$   has the following properties:
\begin{enumerate}
\item
$\FM(J_4^1)$  embeds into $\lt(V)$
for any $V$ of infinite dimension. 
Moreover, the embedding can be chosen such that
any prime quotient has infinite index.
\item $\FM(J_4^1)$ has
infinite height.
\item $\FM(J_4^1)$ has prime quotient 
$y_\top/(y_1+y_2)$, generating the
unique proper congruence
relation $\theta$.  
\item $\FM(J_4^1)/\theta$ is isomorphic to $M_4$.
\end{enumerate}
\end{lemma}

\begin{proof}{ of Theorem 1}.
Given $\phi \in \Sigma$ from Lemma~\ref{2}, consider the presentation
 $\phi^\#$ with  generators $\bar x,x_\bot,x_\top,\bar y,y_\bot,y_\top$
and the relations from $\phi$, $\pi$, and in addition
$x_\top=y_\top$ and  $x_\bot=y_1+y_2$.
Considering a modular lattice $L$ with generators and relations
according to $\phi^\#$,
the following are equivalent in view of Lemma~\ref{3}.
\begin{enumerate}
\item  $x_\bot =x_\top$.
\item $L$ is singleton or $M_4$.
\item $L$ is finite.
\item $L$ is of finite height.
\end{enumerate}
Clearly, if $x_\bot=x_\top$ in every modular lattice
admitting presentation $\phi$ then the same applies
to the presentation $\phi^\#$.
 On the other hand, assume that $\phi^\exists$ is valid in some 
modular lattice.
Given  any  vector space $V$, embed $\FM(J_4^1)$ into $\lt(V)$ as in  
(i) of Lemma~\ref{3} and denote $U=y_1+y_2$.
 By (i) of Lemma~\ref{2} one can evaluate 
$\bar x$ in $\lt(V/U)$ such that $\phi(\bar x,x_\bot,x_\top)$ holds
where $x_\bot=U$ and $x_\top=V$. This results into generators of
a sublattice $L$ of $\lt(V)$ satisfying the relations of $\phi^\#$
and such that $x_\bot \neq x_\top$.
Thus, to decide whether
$x_\bot=x_\top$ for all modular lattices admitting presentation $\phi$
 reduces to deciding 
whether 
(i)--(iv) apply to all
 $L \in \mc{A}$ admitting presentation $\phi^\#$. Undecidability 
of the latter problems follows now from (ii) of Lemma~\ref{2}.
\end{proof}

\begin{corollary} For no quasi-variety  $\mc{A}$ as in  Theorem \ref{1} there is
an
algorithm to decide, given a finite presentation,  whether or not
the lattice freely generated in $\mc{A}$ under that presentation is of finite height.
\end{corollary} 

\end{document}